\documentclass[10pt]{amsart}
\usepackage{amssymb,latexsym}

\usepackage[numbers,sort]{natbib}
\usepackage{graphicx}
\usepackage{geometry}                % See geometry.pdf to learn the layout options. There are lots.
\newenvironment{restate}[2][]%
  {\noindent{\bf\ref{#2}#1.}}{\vspace{\baselineskip}}

\makeatletter
\def\ulabel#1#2{\begingroup
\let\@ldCurrentLabel\@currentlabel
\def\@currentlabel{#2~\@ldCurrentLabel}
\label{#1}\endgroup} \makeatother

\newtheorem{thm}{Theorem}[section]
\newtheorem{lem}[thm]{Lemma}
\newtheorem{cor}[thm]{Corollary}
\newtheorem{prop}[thm]{Proposition}

\newtheorem{que}[thm]{Question}
\newtheorem{conj}[thm]{Conjecture}

\theoremstyle{definition}

\newtheorem{define}[thm]{Definition}

\newtheorem{ex}[thm]{Example}

\DeclareMathAlphabet{\mathbbn}{U}{bbold}{m}{n}

\newcommand{\op}{\operatorname}
\newcommand{\dirlim}{\mathop{\varinjlim}\limits}
\newcommand{\invlim}{\mathop{\varprojlim}\limits}

\newcommand{\N}{\mathbb{N}}
\newcommand{\Z}{\mathbb{Z}}
\newcommand{\Q}{\mathbb{Q}}
\newcommand{\R}{\mathbb{R}}

\newcommand{\B}{\mathbb{B}}

\DeclareMathOperator{\interior}{int}

\title{The Geometry and Fundamental Groups of Solenoid Complements}
\author[G. R. Conner, M. Meilstrup, and D. Repov\v s]{Gregory R. Conner, Mark Meilstrup, and Du\v san Repov\v s}
\begin{document}
\pagenumbering{roman}

\begin{abstract}
A solenoid is an inverse limit of circles.  When a solenoid is embedded in
three space, its complement is an open three manifold.  We discuss the
geometry and fundamental groups of such manifolds, and show that the
complements of different solenoids (arising from different inverse limits)
have different fundamental groups. Embeddings of the same solenoid can give
different groups; in particular, the nicest embeddings are unknotted at each
level, and give an Abelian fundamental group, while other embeddings have
non-Abelian groups.  We show using geometry that every solenoid has
uncountably many embeddings with non-homeomorphic complements.
\end{abstract}

\today

\maketitle

\pagenumbering{arabic} \setcounter{page}{1}

In this paper we study 3-manifolds which are complements of solenoids in
$S^3$.  This theory is a natural extension of the study of knot complements
in $S^3$; many of the tools that we use are the same as those used in knot
theory and braid theory.

We will mainly be concerned with studying the geometry and fundamental groups
of 3-manifolds which are solenoid complements.  We review basic information
about solenoids in \ref{sec-intro}.  In \ref{sec-fund-gp} we discuss the
calculation of the fundamental group of solenoid complements. In
\ref{sec-unknotted} we show that every solenoid has an embedding in $S^3$ so
that the complementary 3-manifold has an Abelian fundamental group, which is
in fact a subgroup of $\mathbb{Q}$ (\ref{unknotted}). In \ref{sec-knotted} we
show that each solenoid has an embedding whose complement has a non-Abelian
fundamental group  (\ref{knotted}).  In \ref{sec-geometry} we take a more
geometric approach, and show that each solenoid admits uncountably many
embeddings in $S^3$ with non-homeomorphic complements (\ref{thm-geometry}).
We achieve this by showing that these complements have distinct geometries
using JSJ theory, and thus by Mostow-Prasad rigidity are distinct manifolds.

\section{Introduction}\ulabel{sec-intro}{section}

A solenoid is a topological space that is an inverse limit of circles. Let
$\{n_i\}$ be a sequence of positive integers, and let $f_i:S^1\to S^1$ be
defined by $f_i(z)=z^{n_i}$, where $S^1$ is thought of as the unit circle in
the complex plane. Then we define the solenoid $$\Sigma(n_i) =\invlim
(S^1,f_i).$$ If the tail of the sequence is $1,1,1,\dots$, then the solenoid
is just a circle. If the sequence ends in $2,2,2,\dots$, then we have what is
called the \emph{dyadic solenoid}, $\Sigma_2$.  We will use the dyadic
solenoid for specific examples throughout this paper.

We note that multiple sequences $\{n_i\}$ can determine the same solenoid, up
to homeomorphism. For instance, we may assume each $n_i$ is prime by
replacing any composite number by the sequence of its prime factors.  We may
also remove any finite initial segment of the sequence, and we may reorder
the sequence (infinitely).  Bing notes that if you remove a finite number of
elements from two sequences so that in the remainders, every prime occurs the
same number of times, then the solenoids are topologically equivalent; he
also says that perhaps the converse is true \cite{bing}.  The converse is
confirmed by McCord \cite{mccord}.  A few other references discussing
solenoids are \cite{dantzig,vietoris,hagopian,chinese}.

As solenoids are obtained via an inverse limit construction of compact topological groups $S^1$, we get the standard result that solenoids are also compact topological groups.  Additionally, it is standard that a solenoid has uncountably many path components, each of which is dense in the solenoid, and also that solenoids are not locally connected, nor are its path components.  However, the path components are fairly nice in that they are bijective images of open arcs.  In particular, there is a continuous bijection from the real line onto each path component.  This bijection however is not a homeomorphism, as small neighborhoods in the solenoid path component are not locally connected.  A lift of a small neighborhood to the real line contains infinitely many small disjoint neighborhoods centered at a collection of points unbounded on the line.

While these standard facts together with the inverse limit construction give some nice properties of solenoids, they do not make it apparent that all solenoids embed in $S^3$.  To see this, we will construct the solenoid $\Sigma(n_i)$ as a nested intersection of solid tori.  Take a solid torus $T_0$ with cross-sectional diameter $d_0$ in $S^3$, using the standard metric from $S^4$.
Embed a solid torus $T_1$ with cross-sectional diameter
$d_1<d_0/2$ inside of $T_0$ that wraps around $T_0$ $n_1$ times.
Continue this process, embedding a solid torus $T_i$ with cross-sectional diameter
$d_i<d_{i-1}/2$ inside of $T_{i-1}$, which wraps around $T_{i-1}$
$n_i$ times. The nested intersection $\bigcap T_i$ is an embedding
of $\Sigma(n_i)$ in $S^3$.  See \ref{embedding} for an example with the dyadic solenoid (where $n_i\equiv2$).

\begin{figure}
\begin{tabular}{cc}
\includegraphics[scale=0.2]{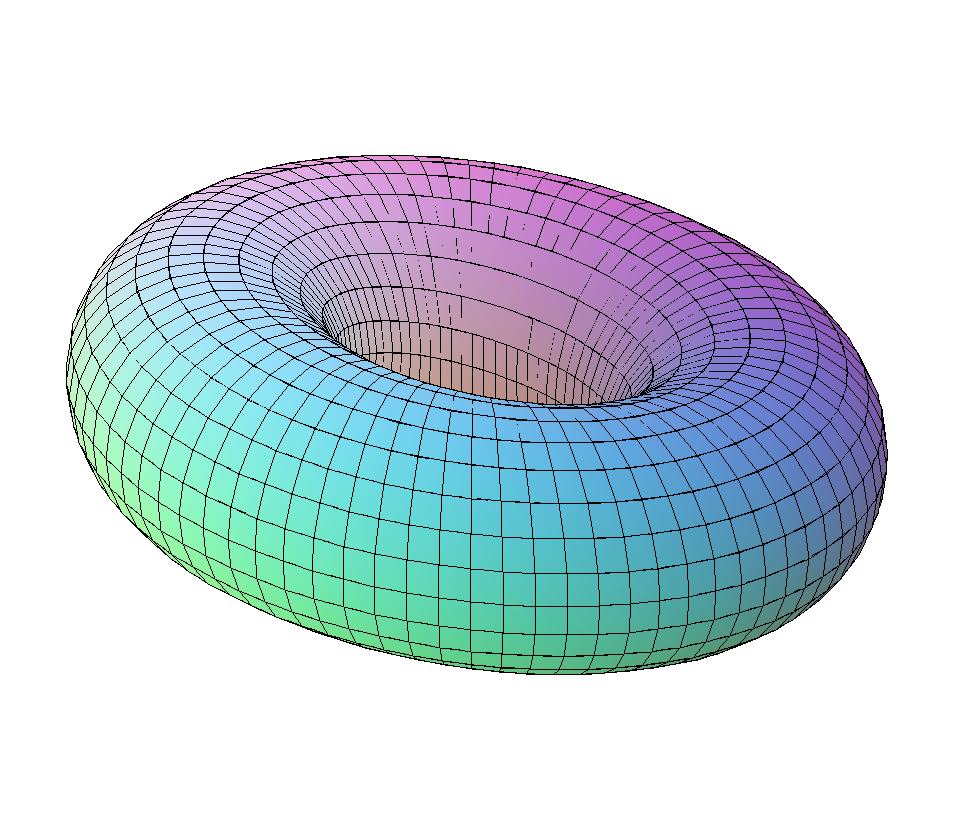}
&
\includegraphics[scale=0.2]{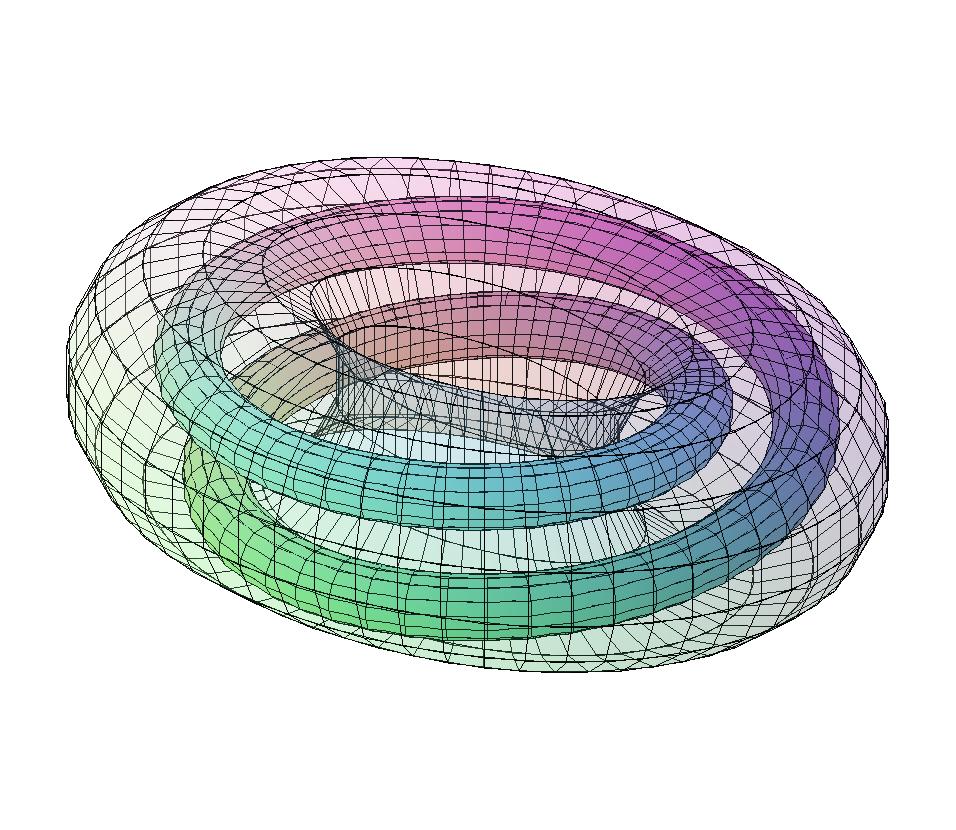}
\\
\begin{tabular}{c}
\includegraphics[scale=0.225]{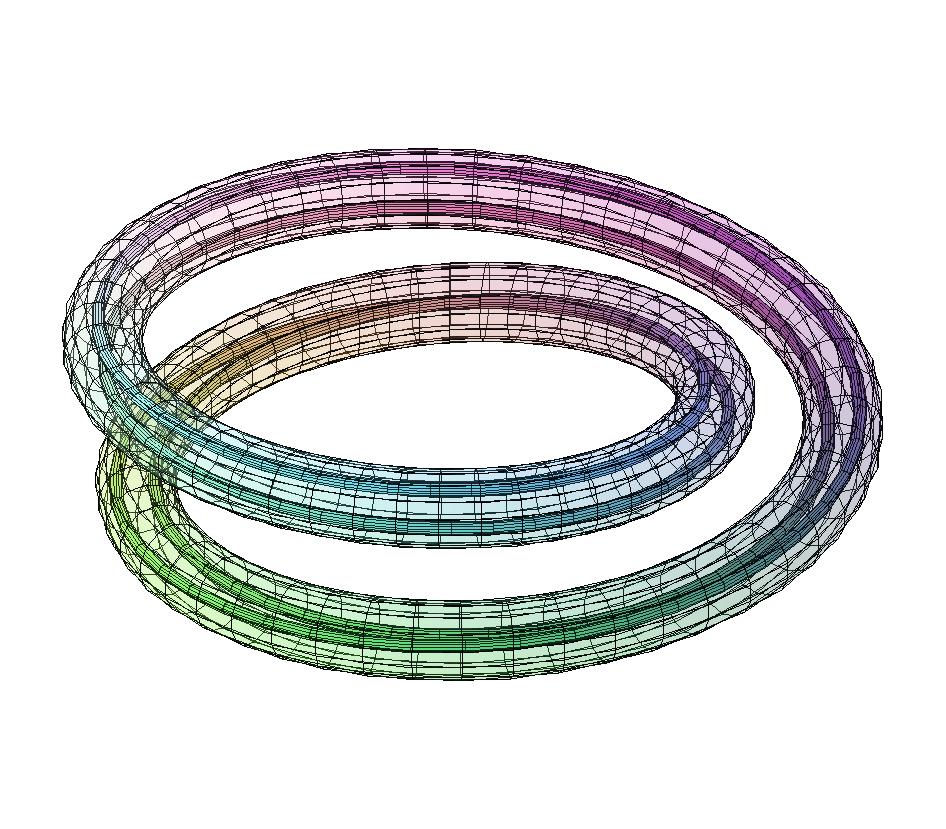}
\end{tabular}
&
\begin{tabular}{c}
\includegraphics[scale=0.4]{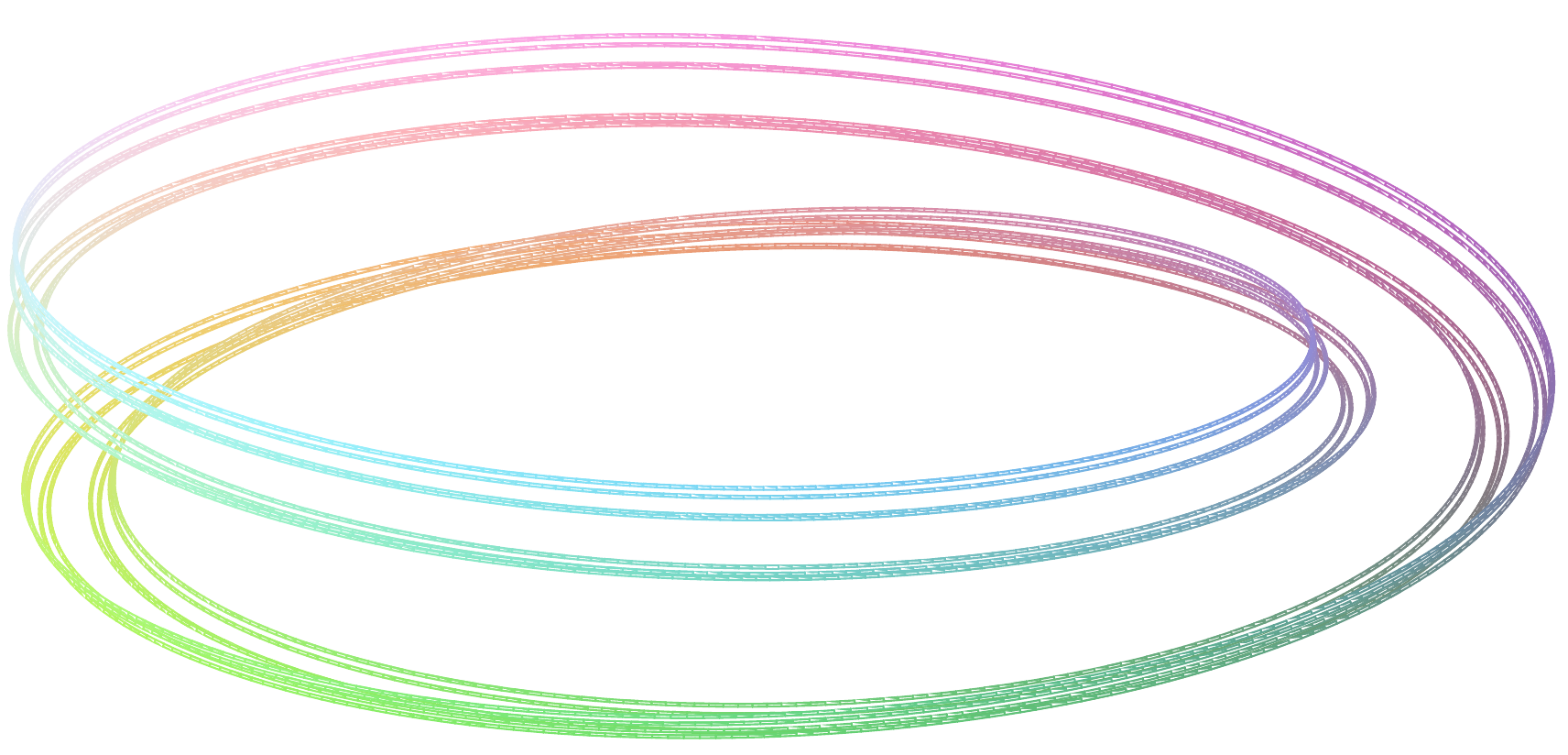}
\end{tabular}
\end{tabular}
\caption[Embedding the dyadic solenoid in $S^3$.]{Embedding the dyadic solenoid in $S^3$.
Begin with a standard unknotted solid torus $T_0$ (top left).  Then embed a second torus $T_1$ inside $T_0$, wrapping around the longitude of $T_0$ twice (top right).  A third torus $T_2$ is shown wrapping twice inside $T_1$ (bottom left).  The solenoid is the infinite intersection of such nested tori (bottom right).
}\ulabel{embedding}{Figure}
\end{figure}

We note that while this nested intersection construction may seem
canonical, there are in fact many ways to embed each $T_i$ inside of
$T_{i-1}$, even if we require that $T_i$ never `folds back' on
itself (i.e.\ $T_i$ is embedded in a monotone fashion inside
$T_{i+1}$). In the simple case where $n_i\equiv2$, $T_i$ can have
any odd number of half twists with itself; when $n_i>2$, there can
be much more complicated braiding. While this does not change the
topology of the solenoid itself, this does change its complement
significantly. This is analogous to knot theory: while every knot is
itself a circle, knot complements are quite different. Thus, we
could consider the study of solenoid embeddings and their
complements as solenoid knot theory. This is also quite related to
braid groups, as each sufficiently nice embedding of $T_i$ into $T_{i-1}$ can
be represented by a braid on $n_i$ strands that gives a transitive
permutation of the strands (otherwise the closed braid will result
in a link with multiple components).  This issue will be discussed further in the following sections, and some diagrams are given in \ref{embedding2}.

All of the embeddings of solenoids that we will consider here will be obtained as nested intersections of solid tori, where each torus is a closed braid in the previous torus.
We note that similar work has been done in \cite{chinese}, where they discuss what they call \emph{tame} embeddings, similar to our braided embeddings.  In \cite{chinese} they are concerned with what they call \emph{equivalent} embeddings, that is, an ambient homeomorphism of $S^3$ taking one embedded solenoid to the other.  We are mainly concerned with the homeomorphism type of the complement, and we believe this to be a distinct question than the notion of equivalent embeddings in \cite{chinese}.
%%dragon
%\\
%
%\textbf{\large{
%DUSAN -- CAN YOU FIND EXAMPLES TO PROVE THE FOLLOWING CONJECTURE???
%}}
%\begin{conj}
%For any fixed solenoid, there are two embeddings that are \emph{not} ambient homeomorphic, but that have homeomorphic complements.
%\end{conj}
%
%\textbf{\large{
%DUSAN --OR CAN YOU ANSWER THE FOLLOWING QUESTION???
%}}
%\begin{q}
%For solenoids, which of the following notions are the same?
%\begin{itemize}
%\item Two embeddings are ambient isotopic.
%\item Two embeddings are ambient homeomorphic (i.e.\ \emph{equivalent} in the sense of \cite{chinese})
%\item Two embeddings have homeomorphic complements in $S^3$.
%\item Two embeddings have complements with isomorphic fundamental groups.
%\item Any other good possible notions of `equivalent embeddings'?
%\end{itemize}
%\end{q}
%
%\vspace{.1in}
%%dragon

It is also interesting to note that solenoids arise in the theory of dynamical systems.  In the case where the sequence $n_i$ is constant, the solenoid can be a hyperbolic attractor of a dynamical system.  These solenoids as attractors were first studied by Smale, and are sometimes called Smale attractors.  A discussion of solenoids as hyperbolic attractors can be found in many books on dynamics, see for instance \cite{hasselblatt_katok}.  A recent result of Brown \cite{brown} shows that generalized solenoids (classified by Williams \cite{williams}) are the only 1-dimensional topologically mixing hyperbolic attractors in 3-manifolds.

\section{Fundamental Groups}\ulabel{sec-fund-gp}{section}

When a solenoid $\Sigma$ is embedded in $S^3$, the complement $\Sigma^c=S^3-\Sigma$ is an open 3-manifold.  As these manifolds are the complement of a non-locally connected space, they have a complicated structure ``at infinity,'' and are not the interior of a compact manifold with boundary.  We will discuss the fundamental groups of such manifolds, which will depend on the particular embedding chosen for the solenoid.
Recall that we are starting with an embedding of the solenoid as a nested intersection of solid tori, each of which is a closed braid in the previous torus:
$$T_0\supset T_1\supset T_2\supset\dots;  \qquad  \Sigma=\bigcap T_i.$$
This gives us that the solenoid complement is an increasing union of torus complements:
$$(S^3-T_0)\subset (S^3-T_1)\subset (S^3-T_2)\subset \dots;  \qquad  \Sigma^c=\bigcup(S^3-T_i).$$
These torus complements are in fact knot complements, where the knots will generally be satellite knots, assuming there is some knotting in the embedding (see the following sections).

The fundamental group of the solenoid complement is then the direct limit of the fundamental groups of the knot complements.  This direct limit is in fact injective, i.e.\ each group injects into the final direct limit, so that it is in fact a union of knot groups, as given by the following lemmas.  Note that our embeddings of solenoids as nested closed braids ensure that the core curve of each torus links the meridional curve of the previous solid torus with linking number $n_i\neq0$.

\begin{lem}\ulabel{jims-lemma}{Lemma}
Suppose that $T_1,T_2$ are solid tori in $\R^3$ with $T_2\subset\interior(T_1)$ and such that the core curve $J$ of $T_2$ links the meridional curve $K$ of $\partial T_1$ having linking number $lk(J,K)\neq0$.  Then the map $\pi_1(\R^3-T_1)\to\pi_1(\R^3-T_2)$ is injective.
\end{lem}
\begin{proof}
Suppose to the contrary that there is a loop $\ell$ in $\R^3-T_1$ that is not nulhomotopic in $\R^3-T_1$ but is nulhomotopic in $\R^3-T_2$.  Let $D:\B^2\to\R^3-T_2$ be a singular disk in $\R^3-T_2$ bounded by $\ell$.

Put $D$ in general position with respect to $\partial T_1$.  By cut and paste, remove all curves of intersection with $\partial T_1$ that are nulhomotopic in $\partial T_1$.  Since the core curve $J$ is not nulhomotopic in $\R^3-T_1$, at least one curve of intersection must remain.

Take such a curve whose preimage is innermost in the domain $\B^2$ of $D$.  This curve is essential in $\partial T_1$ but trivial either in $\R^3-\interior(T_1)$ or in $T_1-T_2$.  The loop theorem thus supplies a nonsingular disk $D'$ whose boundary is nontrivial in $\partial T_1$ but whose interior either lies in $\R^3-T_1$ or in $T_1-T_2$.

In the latter case, $\partial D'$ must be the meridian of $\partial T_1$, hence must link the core curve $J$ of $T_2$, and $D'$ must intersect $J$, a contradiction.  Hence $D'\subset \R^3-\interior(T_1)$, $\partial D'$ must be the longitude of $T_1$, and $T_1$ must be unknotted.

But that implies that $\ell$ is a multiple $m\cdot K$ of the meridional curve $K$ of $\partial T_1$, hence must have linking number $m\cdot lk(J,K)\neq 0$ with $J$, hence cannot be nulhomotopic missing $T_2$, a contradiction.
\end{proof}

\begin{lem}\ulabel{directlimit}{Lemma}
Let $\Sigma=\bigcap T_i$ be the intersection of nested solid tori $T_i$ in $S^3$, such that for each $i$,
the core curve $J$ of $T_{i+1}$ links the meridional curve $K$ of $\partial T_i$ having linking number $lk(J,K)\neq0$.
Then for every $i$, the map $\iota_*:\pi_1(S^3-T_i)\to \pi_1(S^3-\Sigma)$ induced by inclusion
is injective, and $\displaystyle \pi_1(S^3-\Sigma)=\dirlim_i \pi_1(S^3-T_i)= \bigcup_i\pi_1(S^3-T_i)$.
\end{lem}

\begin{proof}
Let $\gamma$ be a nulhomotopic loop in $S^3-\Sigma$, and let $H$ be a nulhomotopy of $\gamma$ in $S^3-\Sigma$.  As $\Sigma$ and the images of $\gamma,H$ are compact, we see that there must be indices $i,k$ such that the image of $\gamma$ lies in $S^3-T_i$, and the image of $H$ lies in $S^3-T_{i+k}$.  As long as $k>0$, we have $\op{im}\gamma\subset S^3-T_i \subset S^3-T_{i+k-1}$, and
%Since the core curve of each $T_{i+k}$ links with the meridional curve of $\partial T_i$ with linking number %$\ds\prod_{j=i+1}^{i+k} n_j\neq0$
%$\prod_{j=1}^{k} n_{i+j}\neq0$,
we may use \ref{jims-lemma} to see that $\gamma$ is nulhomotopic in $S^3-T_{i+k-1}$.  Repeating this process $k$ times shows that $\gamma$ is in fact nulhomotopic in $S^3-T_{i}$.
Thus each $\pi_1(S^3-T_i)$ injects into $\pi_1(S^3-\Sigma)$, and the lemma is proven.
\end{proof}

%zxcvasdf
%
%apparently my proof below was false...
%\begin{proof}
%We first note that each $T_{i+1}$ intersects every meridional disk $D$ in $T_i$, i.e.\ every disk bounded by an essential curve on $T_i$ (which will be a meridian).  If not, then cutting along $D$ gives us $T_{i+1}$ inside a 3-ball, so that $T_{i+1}$ could not be essential.
%
%Let $\gamma$ be a loop in $S^3-\Sigma$, and let $H$ be a nulhomotopy of $\gamma$ in $S^3-\Sigma$.  As $\Sigma$ and the images of $\gamma,H$ are compact, we see that there must be indices $i,k$ such that
%$\gamma$ lies in $S^3-T_i$, and the image of $H$ lies in $S^3-T_k$.
%Assume the map $H$ is in general position with respect to the tori $\{T_j\}$.  Then the image of $H$ intersects $T_{k-1}$ in a collection of simple closed curves.  As the image of $H$ misses $T_k$, none of these curves bounds an essential loop on $T_{k-1}$, and we may adjust $H$ to a homotopy in $S^3-T_{k-1}$.  Iterating this process, we see that $H$ is equivalent to a homotopy in $S^3-T_i$.  Thus a loop is nulhomotopic in the complement if and only if it is nulhomotopic in the stage where it appears, thus we get an injective direct limit, or union, of fundamental groups.
%\end{proof}

%%index number change -start 1
Recall that $S^3$ is the union of two solid tori; we will embed a solenoid into one of these.  In order to calculate the fundamental group of the solenoid complement, we will cut the space along the tori $\{T_i\}$, to get pieces $T_{i-1}-T_{i}$ that are each a solid torus minus a braid, together with one piece that is simply a solid torus (the initial complementary solid torus in $S^3$).  We will calculate the fundamental group of each piece, and then use the Seifert Van Kampen Theorem to get relations between the pieces, as the outer torus of one piece is the inner torus, or braid, in the previous piece.  The union of all of these groups and the Van Kampen relations will give a presentation for the fundamental group by \ref{directlimit}.

The fundamental group $\pi_1(T_{i-1}-T_{i})$ can be calculated by considering the space $T_{i-1}-T_{i}$
as a mapping cylinder over an $n_i$-punctured disk.
Thus the group has the form
$$\pi_1(T_{i-1}-T_{i})= \big\langle t,x_1,\dots,x_{n_i} ~\big|~ t^{-1}x_k t = w_k(x_1,\dots,x_{n_i}) \big\rangle.
$$
The $x_i$'s represent free generators of the fundamental group of a punctured disk, and $t$ represents the longitude of the outer torus $T_{i-1}$.  Here $w_k$ is some word in the $x_j$'s, depending on the embedding (braiding)
of one solid torus inside the previous.  We note that the for each $k$, if strand $k$ attaches to strand $m$ in the closed braid, then the word $w_k$ is a conjugate of $x_m$.  See \ref{mappingcylinder}.

\begin{figure}
\includegraphics[scale=1.0]{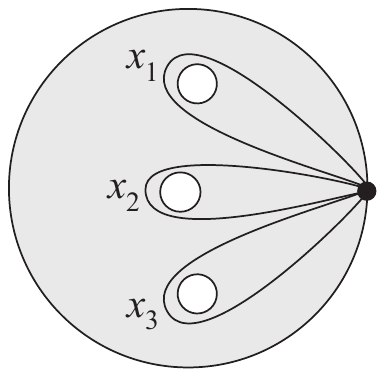}
\hspace{.5in}
\includegraphics[scale=.8]{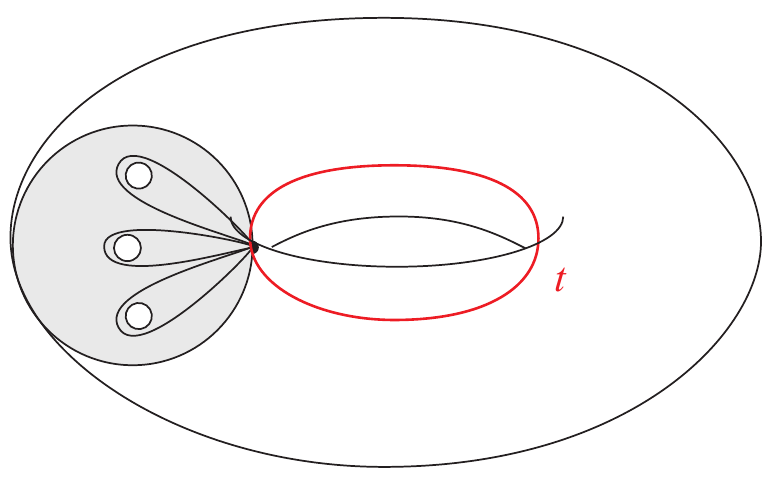}
\caption{Generators for $\pi_1(T_{i-1}-T_{i})$.}\ulabel{mappingcylinder}{Figure}
\end{figure}

We will apply Seifert Van Kampen to get the relations connecting the various pieces.  As such, we need some notation to differentiate the generators from each piece $(T_{i-1}-T_{i})$. The loop $x_{(i)k}$ will be a meridian of the torus $T_i$, or equivalently a loop going around one of the strands of the braid inside of $T_{i-1}$.  The loop $t_{(i)}$ will be a longitude of $T_i$.
Thus the variables $x_{(i)k},t_{(i-1)}$ correspond to the fundamental group of
the piece $(T_{i-1}-T_{i})$ as discussed previously.
The word $v_{(i)}\big(\{x_{(i)k}\}\big)$ is determined by the embedding, relating the longitudes $t_{i-1},t_i$ of the tori $T_{i-1},T_{i}$.
With this notation in place, we use Van Kampen's Theorem to get relations such as
$$x_{(i-1)1}=\prod_{k=1}^{n_{i}} x_{(i)k}, \qquad
t_{(i)}=t_{(i-1)}^{n_i}v_{(i)}(x_{(i)1},\dots,x_{(i)n_i}) .%x_1^m.
$$

Putting all of this together, we get an infinite presentation for $\pi_1(\Sigma^c)$.
The generators are $t_{(i)},x_{(i)k}$ from each level $i$, with $k=1,\dots, n_i$.
The relations come from each level and Van Kampen's theorem.  Recall that the words $w_{(i)k},v_{(i)}$ are dependent on the braided embedding of one torus in the previous.
Also note that $t_{(0)}=e$, since the longitude of $T_0$ is trivial in $S^3$, as its complement is simply a solid torus.
$$\pi_1(\Sigma^c)=\Big\langle t_{(i)},x_{(i)k} ~\Big|~  t_{(i-1)}^{-1}x_{(i)k} t_{(i-1)} = w_{(i)k}(\{x_{(i)k}\}), t_{(0)}=e \qquad \qquad
$$
$$\qquad \qquad \qquad  \qquad ~~ x_{(i-1)1}=\prod_{k=1}^{n_{i}} x_{(i)k}, t_{(i)}=t_{(i-1)}^{n_i}v_{(i)}(\{x_{(i)k}\}) \Big\rangle
$$

\begin{ex}[Dyadic Solenoid]\ulabel{dyadicpresentation}{Example}
In the case of the dyadic solenoid with defining sequence $n_i\equiv2$, our presentation for $\pi_1$ simplifies.  There are only two $x_{(i)k}$'s at each level $i$, and since $x_{(i-1)1}=x_{(i)1}x_{(i)2}$, we do not actually need any of the generators $x_{(i)2}$.  If we let $z_i=x_{(i)1}$ be the meridian of $T_i$, and $s_i=t_{(i)}$ the longitude of $T_i$, we then get a simplified presentation, where $R$ represents relations dependent on the braiding:
$$\pi_1=\big\langle s_i,z_i ~\big|~ [s_i,z_i]=e,~ R,~ s_0=e \big\rangle
$$
\end{ex}

%%index number change -end 1

\section{Unknotted Solenoids}\ulabel{sec-unknotted}{section}

We define an unknotted embedding of any solenoid in $S^3$, and discuss the fundamental group of its complement. We will discuss knotted embeddings in the next section.

\begin{define}
An embedding of a solenoid as a nested intersection of solid tori
$T_i$ is \emph{unknotted} if each $T_i$ is unknotted (in $S^3$).
\end{define}

We will show that every solenoid has an unknotted embedding, and that the complement of an unknotted embedding has Abelian fundamental group.

While there are many braids on $n$ strands that give the unknot, the simplest is probably $b(n)=\prod
\sigma_i=\sigma_1\sigma_2\dots\sigma_{n-1}$, in terms of the standard braid generators $\sigma_i$.  Note that we could just as easily have reversed the order, or used inverses ($\sigma_i^{-1}$). These closed braids just wrap around $(n-1)$ times without any crossings, and then take the first (or last) strand over (or under) all of the other strands.

There is an obvious way to try to embed the next level in this one:
thicken each strand to a tube, draw $n_{i}$ parallel strands in
each tube (crossing all of the strands in one tube over all of the
strands in another when the tubes cross), and put the braid for the
next level in one tube in some portion where there are no crossings
of the tubes.
Unfortunately, this obvious way to iterate this process does not
produce an unknot. This is due to the fact that there is some
inherent twisting in each stage that will show up in the following
stages, if not dealt with carefully.

As an example, consider just two levels, where both $n_1,n_2= 2$. On the
first level, we have two strands, and we will use $b(1)=\sigma_1$ as our
braid (if we had chosen to use inverses for $b(n)$ the following works out
similarly).  On the next level, we have four strands.  If we start
with $\sigma_1$, and then just follow the previous stage with the
strands parallel to each other, the resulting knot is actually a
trefoil, rather than the unknot.  However, if you instead start with
$\sigma_1^{-1}$ (or even $\sigma_1^{-3}$), you do get the unknot. It
is more enlightening to say that if you begin with
$\sigma_1^\epsilon \sigma_1^{-2}$ you get the unknot, if
$\epsilon=\pm1$. This is true because unwinding the doubled
structure from the first level cancels out the $\sigma_1^{-2}$,
leaving $\sigma_1^\epsilon$, which is the unknot. We leave it to the
reader to verify that the given braids yield the specified knots.  These braids and the resulting knots are shown in \ref{embedding2}.

\begin{figure}
\begin{tabular}{ccc}
\begin{tabular}{c}
\includegraphics[scale=0.5]{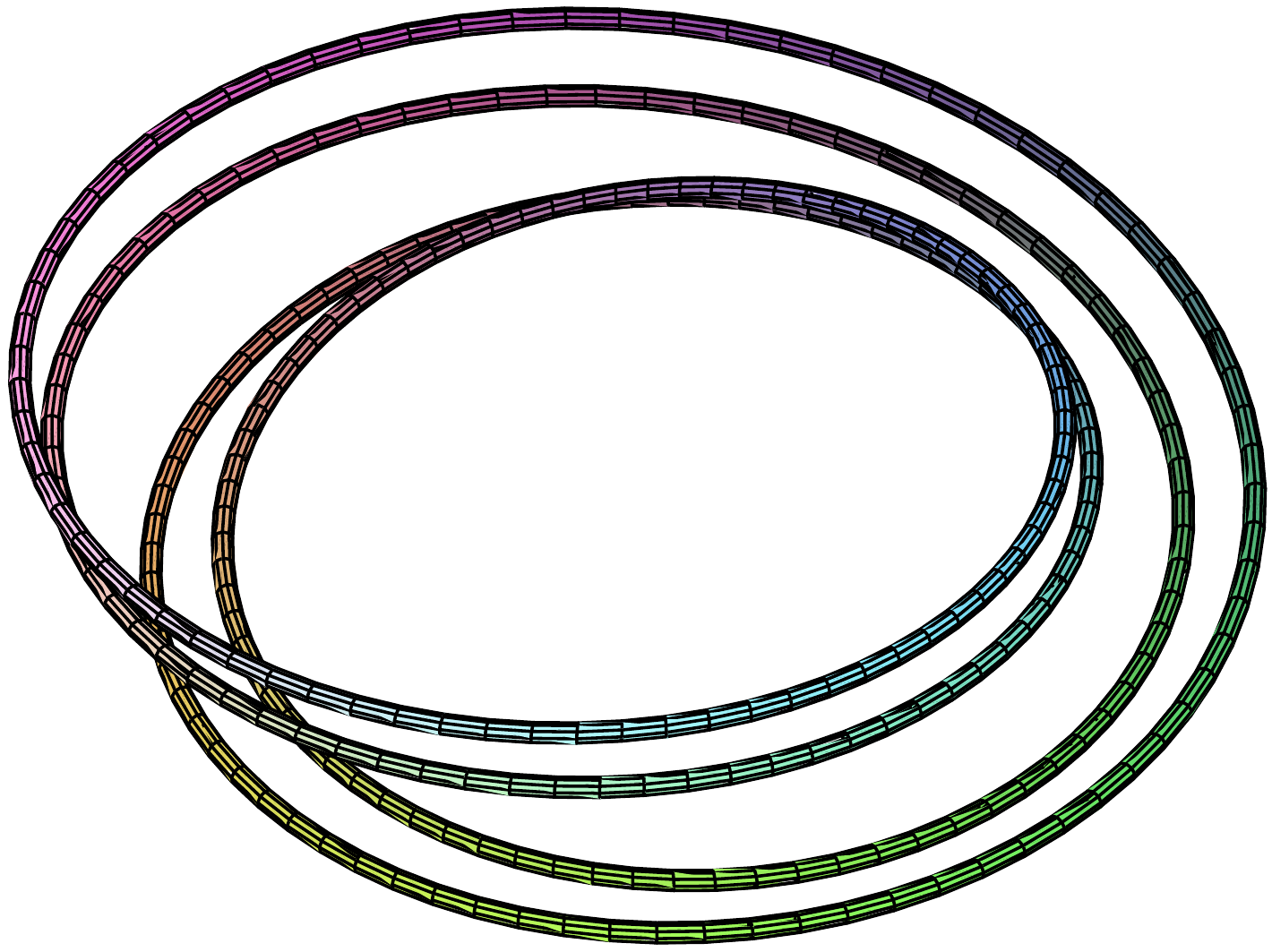}
\end{tabular}
&\qquad\qquad\qquad&
\begin{tabular}{c}
\includegraphics[scale=0.4]{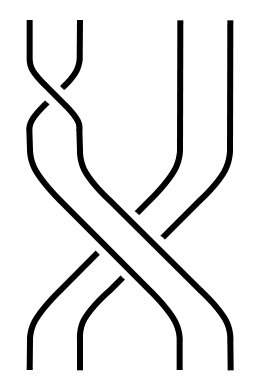}
\end{tabular}
\\\\\\
\begin{tabular}{c}
\includegraphics[scale=0.38]{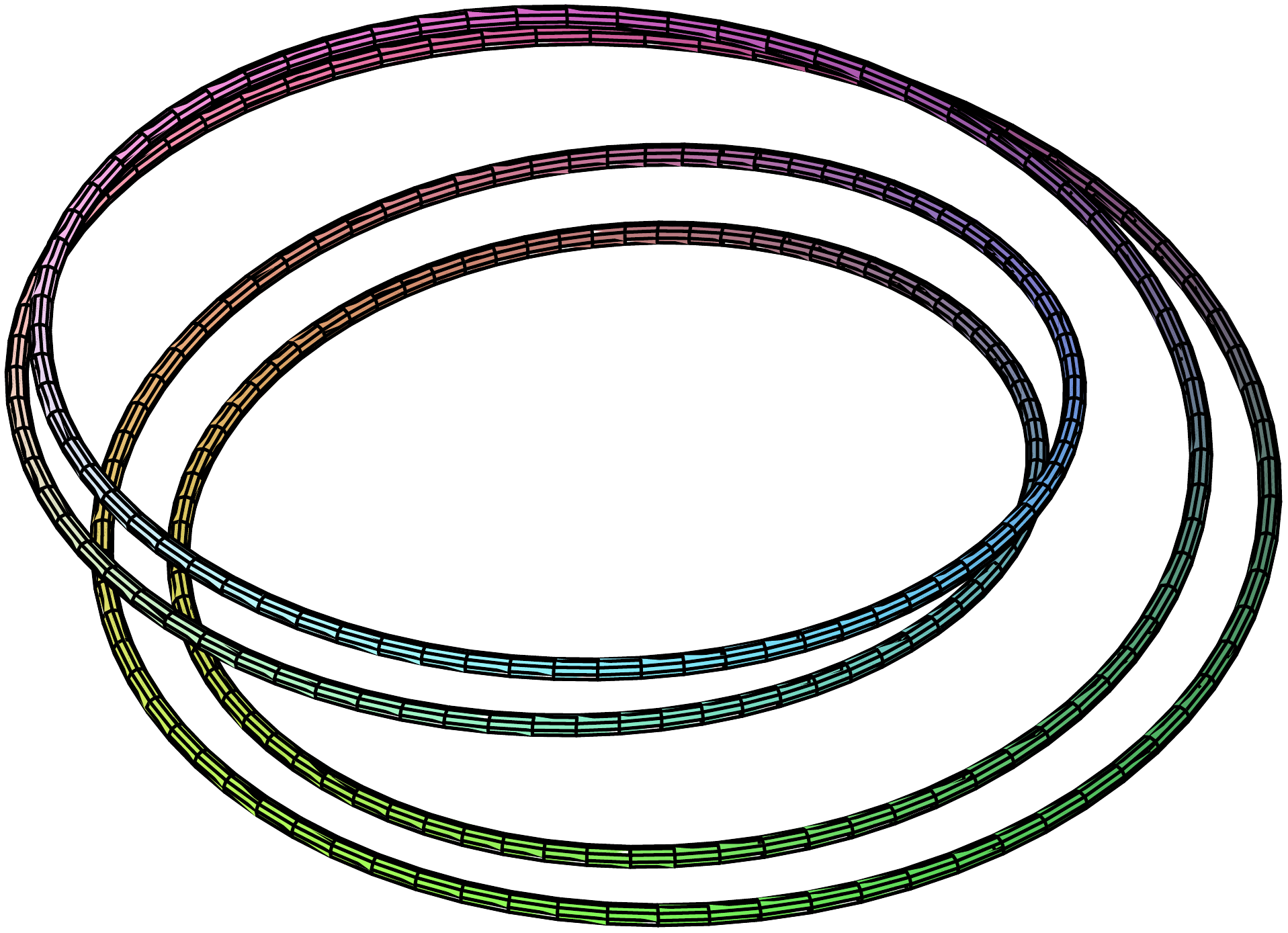}
\end{tabular}
&\qquad\qquad\qquad&
\begin{tabular}{c}
\includegraphics[scale=0.4]{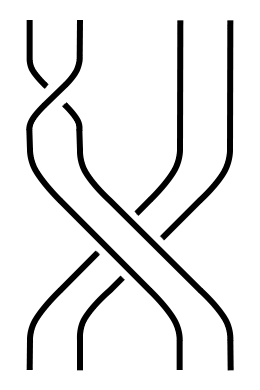}
\end{tabular}
\\\\\\
\begin{tabular}{c}
\includegraphics[scale=0.38]{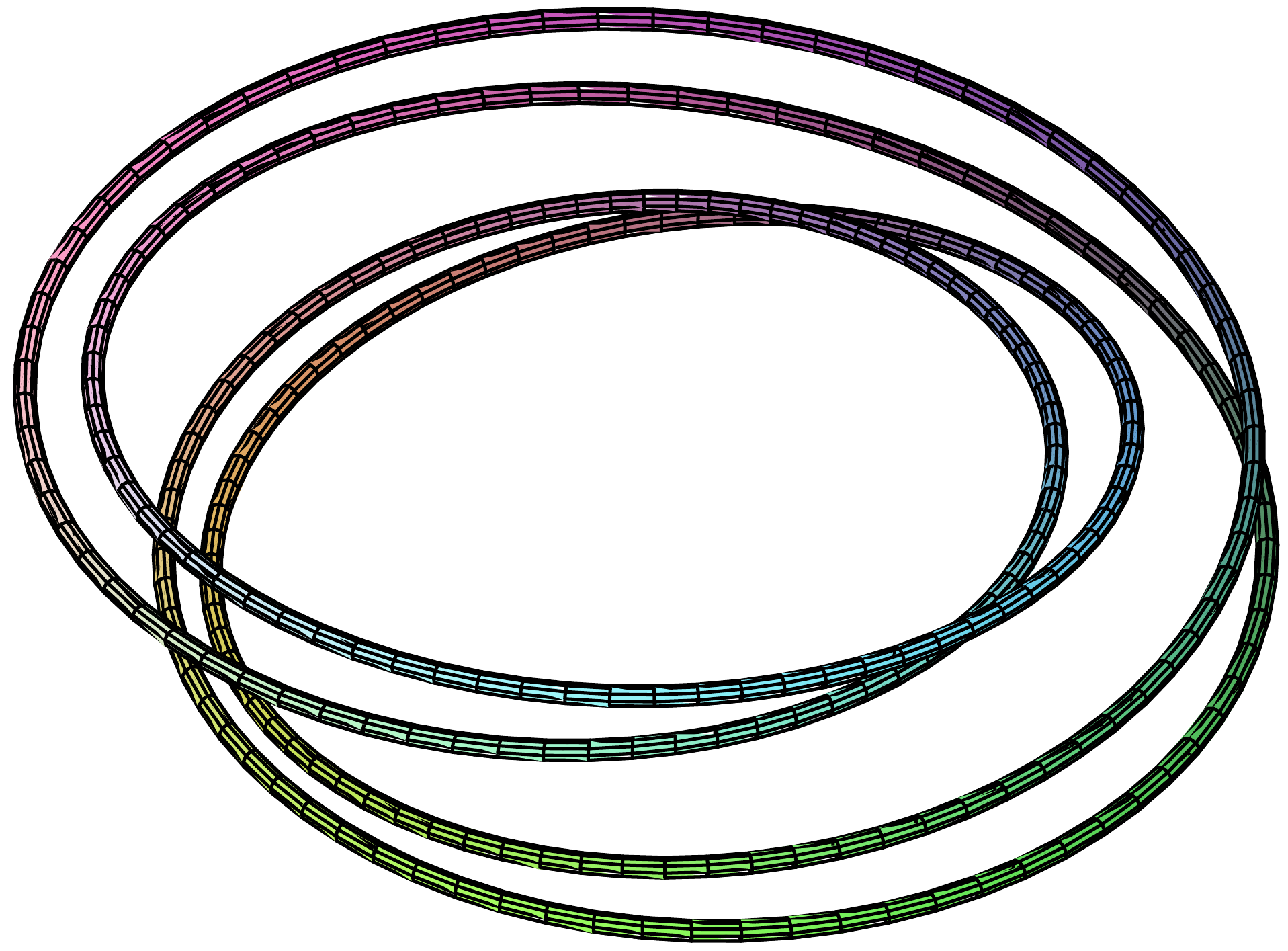}
\end{tabular}
&\qquad\qquad\qquad&
\begin{tabular}{c}
\includegraphics[scale=0.4]{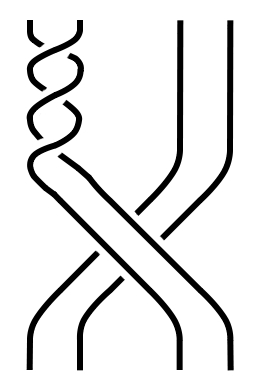}
\end{tabular}
\end{tabular}
\caption[Two levels of the dyadic solenoid embedded as the trefoil and the unknot.]{Two levels of the dyadic solenoid embedded as the trefoil (top row), the unknot (middle row), and another version of the unknot (bottom row). The diagrams on the right show the corresponding braids.}\ulabel{embedding2}{Figure}
\end{figure}

Even though the obvious method does not work, it is possible to keep
track of the twists in such a way to get an unknotted embedding of
the solenoid.  This basically amounts to adding some amount of extra full twists (of all the strands) to correct for the twisting from the previous level.  In the case of the braids $b(n)$ which we have chosen above, this ends up being precisely $(n-1)$ full twists. The case of the dyadic solenoid with $n_i\equiv2$ amounts to adding one full twist, and three levels of this embedding are shown in \ref{unknotbraids}.
This twisting will also become apparent as we discuss the algebraic structure later, particularly in the example of the dyadic solenoid (see \ref{dyadicunknot}).

We note here that this process of constructing unknotted braids can be continued indefinitely, thus providing an embedding of the solenoid.  At first there may seem to be a difficulty due to the fact that our embedding requires nested tori, while our braid construction here does not obviously satisfy that condition.  However, one can check that each level of our braid construction does nicely embed in the previous. For example, in \ref{unknotbraids}, taking a tubular neighborhood of the four strands on the left and the four strands on the right gives a 2-braid with one crossing, just as in the top left single crossing in the diagram.  Also, taking a neighborhood of two strands at a time gives a 4-braid that is the same as the top left portion of the diagram (above the full twist on four strands).

We briefly describe one other way to see that this always works, even for more complicated braids that may represent the unknot.  Start with one level embedded in $S^3$ as a torus, which has an ambient isotopy $h$ to the standard unknotted torus.  Embed the next desired level in the interior of the standard unknotted torus, and then composing with $h^{-1}$ gives the desired embedding of the next level.  While this process works for any braid representation for the unknot, our chosen simple braids $b(n)$ admit a formulaic description.  We will proceed using our chosen braids $b(n)$, and at the end of the section we will comment on the general case.

\begin{figure}
\begin{center}
\includegraphics[scale=.55]{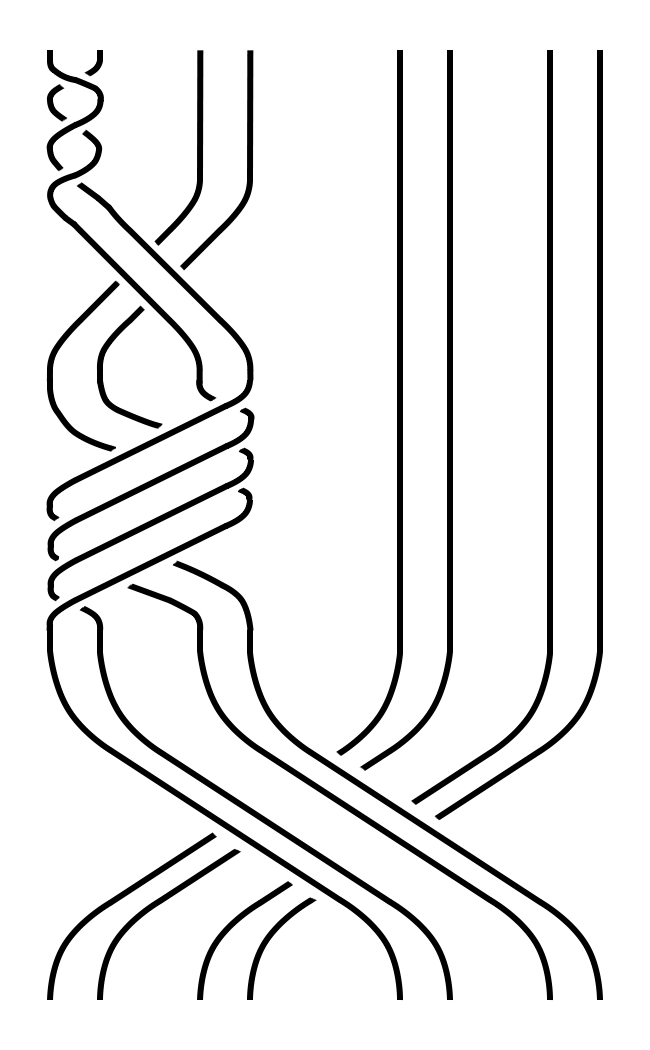}
\end{center}
\caption{Multiple levels of unknotted braids.}\ulabel{unknotbraids}{Figure}
\end{figure}

To compute the fundamental group of the solenoid complement, we
first compute the fundamental group of a solid torus minus the chosen
closed braid $b(n)=\prod\sigma_i$.
As in the previous section, we can present the fundamental group of this piece as
%We note that this can be thought
%of as a mapping cylinder over a disc with $n$ punctures.  Thus the
%fundamental group is
$G(n)=\langle t, x_1, x_2,\dots,x_n ~|~ R\,\rangle$, where $x_i$ represents the loop going around the
$i^{\text{th}}$ puncture once, $t$ represents the longitude of the solid torus.
The relators $R$ are determined by the braid $b(n)$ as follows: for $i>1$, we have
$t^{-1}x_i t = x_{i-1}$, together with the relation
$t^{-1}x_1 t = x_1 x_2\dots x_{n-1}x_n x_{n-1}^{-1}x_{n-2}^{-1}\dots x_1^{-1}$.
Note that if we kill $t$ (i.e. set $t=e$), then these relators become $x_i=x_{i-1}$, and thus the
quotient $G(n)/\hspace{-2pt}\ll t\gg\, =\langle x_1\rangle = \mathbb
Z$. This should be expected, as this is equivalent to gluing in a
solid torus to get $S^3$ minus the braid $b(n)$, which was the
unknot. In the following, it will be convenient to set $x_0=\prod
x_i$, which satisfies the relation $t^{-1} x_0 t= x_0$.

%%index number change -start 2
Now we consider the Seifert Van Kampen relations.  As we only are looking at two levels for the moment, we will denote the elements of
the inner piece with `primes,' (as in $x_k'$ compared to $x_k$ for the outer piece), to avoid the more cumbersome notation $x_{(i)k}$ used previously for the complete presentation of the solenoid complement fundamental group.
Then the relations determined by the meridian and longitude of the intersection torus are
$x_1=x_0'=\prod x_k'$, and $t'=t^{n_1} w(x_k)$, where $w$ is some
word in the $x_k$'s.

This is where the issue of twisting comes into play.  By considering
a diagram, one can see that $w=x_0$ does work (it's useful to
remember that $x_0$ commutes with $t$). While not every other word
in $x_k$ will work, we can match this longitude of the intersection
torus with any longitude $t'(x_0')^m$ of the inner torus, perhaps
wrapping around more (or fewer) times than we think we should. As
$x_0'=x_1$, we see that we can append any number of $x_1$'s at the
end of $w$. In order to get the unknot at this second level, we
choose $w=x_0 x_1^{-n_1}$. Thus when we set $t=e$, we get that
$x_i=x_j$, so that $x_0=x_1^{n_1}$ and $w=e$. Then $t'=e$, and we
similarly have $x_i'=x_j'$. Also, $x_1=x_0'=\prod
x_i'=(x_1')^{n_2}$. Thus the fundamental group is generated by
$x_1'$, where the generator $x_1$ from the previous step satisfies
$x_1=(x_1')^{n_2}$. Therefore the second stage is unknotted, being a
knot with fundamental group $\mathbb Z$, and the fundamental group from the first stage embeds via the map $\Z\to\Z:1\mapsto n_2$.

We can continue this process of inserting unknotted solid tori
$T_i$, and we get that each fundamental group
$\pi_1(S^3-T_i)$ is cyclic. If we call the generators from two
consecutive stages $a,a'$, then we have $a=(a')^{n_k}$. Thus we see
that the fundamental group of the complement of the unknotted
solenoid $\Sigma(n_i)$ is the direct limit $G(n_i)=\dirlim (\mathbb Z,
f_i:1\mapsto n_i)$.  This group can be described more directly as follows, since we are allowed to divide by any of the $n_i$:
$$G\big(\{n_i\}\big)=\left\{\frac pq\in\mathbb Q ~\middle|~ q =% \text{ divides }
\prod_{i=1}^k n_i \text{ for some } k \right\}.
$$

The element 1 in this group represents the meridian loop of the
initial torus $T_0$ in the construction, and $1/n_1$ represents the meridian
loop of the torus $T_1$, or going around one strand of the braid in
the first level $(T_0-T_1)$.  At each stage we can divide by $n_i$, and in
general $1/\left(\Pi^k n_i\right)$ represents a loop going around a
strand of the braid on the $k^{th}$ level $(T_{k-1}-T_{k})$, or equivalently, a
meridian of the torus $T_k$. Since any loop can only come to
within a finite (non-zero) distance of the solenoid, this gives us
all loops in the fundamental group.

%%index number change -end 2

\begin{ex}[Dyadic Solenoid]\ulabel{dyadicunknot}{Example}
If $\Sigma $ is the dyadic solenoid with defining sequence $n_i\equiv2$, then this
tells us that the fundamental group is the direct limit
$\dirlim(\Z,2)$, which is just the dyadic rationals $G=\{p/2^k\}$.

This can also be seen from the presentation as given in \ref{dyadicpresentation}.  For each level being unknotted we get the following presentation, where we have filled in the relations $R$ from the presentation earlier.
$$\pi_1=\Big\langle s_i,z_i ~\Big|~ [s_i,z_i]=e,~ {s_i}^{-1}z_{i+1}s_i={z_{i+1}}^{-1}z_i,~ s_{i+1}={s_i}^2z_i{z_{i+1}}^{-2},~ s_0=e \Big\rangle
$$
Notice that on any level, if $s_i=e$, then ${z_{i+1}}^2=z_i$, and then $s_{i+1}=s_i^2=e$.
Thus this group becomes
$\langle z_i ~|~ z_{i+1}^2=z_i \rangle = \dirlim (\mathbb Z,2).
$
\end{ex}

Similarly, for an $n$-adic solenoid, where $n_i\equiv
n$, we get the (non-complete) $n$-adic rationals $\{p/q ~|~
q=n^k\}$. In general, the group $G$ can be any non-trivial subgroup of $\Q$.  We characterize the subgroups of $\Q$ in \ref{subgroups of Q}; we restate the lemma here for convenience, and give a proof in the appendix.  We then describe how to achieve those as the fundamental group of a specific solenoid complement.

Note that for additive subgroups of $\Q$, multiplication by a constant is an isomorphism, so that we may assume that the subgroup contains 1.  In the lemma, the numbers $k_i$ represent the number of times (plus 1) that the prime $p_i$ is allowed to appear in the denominators of the subgroup elements.
\\

\restate{subgroups of Q} {\emph{%
Let $\{k_i\}$ be a sequence in $\mathbb N\cup\infty$.  Define }
$$Q\big(\{k_i\}\big)=\left\{\frac pq\in\mathbb Q ~\middle|~ q =%\text{ divides }
\prod_{i=1}^m
p_i^{n_i} \quad\text{ for some } n_i< k_i \text{ and some } m %\right{\phantom{)}}   %
\right\}
$$
\emph{where $p_i$ denotes the $i^{th}$ prime number.}

\emph{
Then $Q\big(\{k_i\}\big)$ is a subgroup of $\mathbb Q$ containing 1. Furthermore,
every subgroup $G\leq\mathbb Q$ containing 1 is equal to $Q\big(\{k_i\}\big)$
for some sequence $\{k_i\}$.}
}\\

For a solenoid with defining sequence $\{n_i\}$, the fundamental group is $G\big(\{n_i\}\big)$ as mentioned above, which can also be described as the subgroup $Q\big(\{k_j\}\big)$ from \ref{subgroups of Q} by setting $k_j$ to be one more than the cumulative number of times the $j^{th}$ prime occurs as a factor in the sequence $\{n_i\}$ (where $k_j$ might be infinite).  For example, if the sequence $\{n_i\}$ begins with $2,4,6,8,5,\dots$, where the tail of the sequence consists of odd numbers, then for $i=1$, we have $p_i=2$, and $k_i=1+(1+2+1+3+0)=8$, as we add one to the sum of the powers of 2 that appear in the $n_i$.

From this, it is now easy to see that given any subgroup
$Q(k_j)\leq\mathbb Q$, there is a solenoid $\Sigma$ and an unknotted
embedding into $S^3$ such that $\pi_1(S^3-\Sigma)=Q(k_j)$. The
defining sequence $\{n_i\}$ can be chosen in various ways, but the homeomorphism type of the solenoid
described is uniquely determined.  One construction that will always
work is as follows:
$$n_i=\prod_{j=1}^i p_j^{m_{ij}},\quad\text{where } m_{ij}=1 \text{ if } i-j<k_j-1
\text{ and 0 otherwise}.
$$

This construction may have some $n_i=1$, and these may be removed if
the tail of the sequence $n_i$ is not identically 1.  In the case where the $n_i$ are eventually 1,
the subgroup $Q\big(\{k_j\}\big)$ is cyclic ($\cong \mathbb Z$), and the
required solenoid is the circle $S^1$ (with $n_i\equiv1$).  The circle is not always considered a solenoid, being trivially so.  If not considering $S^1$ to be a solenoid, then any subgroup of $\Q$ that is neither $\{0\}$ nor $\Z$ may be obtained as the fundamental group of a solenoid complement.

As mentioned earlier, any finite segment of $\{n_i\}$ does not change
the solenoid $\Sigma(n_i)$. It also does not change the fundamental
group of the unknotted complement. If $G\big(\{n_i\}\big)$ is the group for the
sequence $\{n_i\}$, and $G\big(\{n_i\},k\big)$ is the group where we start the
sequence at $i=k$, then we have the isomorphism $\varphi:G\big(\{n_i\}\big)\to
G\big(\{n_i\},k\big)$ defined by $\varphi(x)=x\cdot\left(\prod_{i=1}^{k-1}n_i \right)$.

Also notice that any reordering of $\{n_i\}$, or replacing a term
$n_j$ by a sequence of its prime factorization, will also not change
the group (here the isomorphism is the identity map).

In the previous discussion, we considered a particular unknotted
embedding, based on a choice of braids $b(n)$ that give the unknot.
There are obviously many other choices of braids that give the
unknot; for example, the combined braid from the first and second
stages described above is an unknot on $n_1n_2$ strands, which
differs from our chosen $b(n_1n_2)$.  However, the results stated
above still hold.  Given any unknotted embedding, we have
$\pi_1(S^3-\Sigma)=\dirlim \pi_1(S^3-T_i)=\dirlim \mathbb Z$, as each $T_i$ is unknotted.
The bonding maps are still $f_i:1\mapsto n_i$, resulting in the same fundamental group.

While, for a given solenoid, the fundamental group of the complement
of any unknotted embedding is the same, one may ask the following
\begin{que}
Are all unknotted embeddings of a given solenoid equivalent?
\end{que}
Here we might take equivalent to mean that there is an ambient
isotopy, or perhaps ambient homeomorphism (possibly orientation
preserving) between the two embeddings, or perhaps just requiring that the complements in $S^3$ be homeomorphic.

As noted above, changing any finite segment of the $n_i$'s does not change the fundamental group, but additionally in this case the complements are homeomorphic, as we may `unwind' the first $k$ levels of unknotted tori, with an ambient isotopy.  Similarly, any finite reordering of the $n_i$ or replacement by factorizations or products will also give ambient isotopic embeddings.  Additionally, changes in the unknotted embeddings chosen at finitely many levels will also give ambient isotopic embeddings.  To ensure equivalent embeddings, we only need to require that there are only finitely many changes in the sequence $\{n_i\}$ and in the unknotted embeddings.  If there are infinitely many of these changes made, it is no longer clear whether this changes the homeomorphism type of the complement.

It seems likely that infinitely many changes will result in different complements, or at least embeddings that are not ambient isotopic, and that there should be uncountably many inequivalent unknotted embeddings for any solenoid.
\begin{conj}\ulabel{unknottedconj}{Conjecture}
For every solenoid, there are uncountably many inequivalent unknotted embeddings in $S^3$.
%, i.e.\ embeddings with non-homeomorphic complements.
\end{conj}

We summarize the results of this section in the following theorem:

\begin{thm}\ulabel{unknotted}{Theorem}
%\begin{itemize}
%\item
For any solenoid $\Sigma$, there exists an embedding
$\Sigma \subset S^3$ such that $\pi_{1}(S^3 - \Sigma)$ is
Abelian, and in fact a subgroup of $\mathbb Q$.

%\item
Furthermore, for every nontrivial subgroup $G\leq (\mathbb Q,{+})$,
there exists a solenoid $\Sigma$ and an embedding $\Sigma \subset S^3$
such that $\pi_{1}(S^3 - \Sigma)\cong G$.
%\end{itemize}
\end{thm}

\section{Knotted Solenoids}\ulabel{sec-knotted}{section}

In the previous section, we took care to ensure that each torus in
the nested intersection construction was unknotted in $S^3$. First,
we used a braid $b(n)$ that represents the unknot, and then we took
care how we glued in the next stage, with respect to twisting.
Relaxing these conditions, we will consider any braid $b$ on $n$
strands that is transitive on the strands; transitivity gives us a
knot instead of a link.% -- we might discuss links later.

Again, the fundamental group of a solid torus minus this closed braid will have the form
$G(b)=\langle t, x_1,\dots, x_n ~|~ R \rangle$. The relators in $R$ are of the form $t^{-1}x_it=w_i$, where the word $w_i$ can be determined directly from the braid.
%%I could give a fairly explicit description based on the braid
%%generators $\sigma_i$, but I don't care to at the moment.
%Presently, however,
We only mention here that if the braid $b$ sends strand $i$ to strand $j$,
then the corresponding relator has the form $t^{-1}x_it=g^{-1}x_jg$,
where $g$ is some word in the $x_k$'s, dependent on the braiding.  Then due to the transitivity, we see that after Abelianization, the relators give $x_i=x_j$ for all $i,j$.

To connect two such tori, we need the extra relations
$x_1=x_0'=\prod x_i'$, and $t'=t^{n_1}w(x_i)$.  By careful
consideration of a braid diagram, one can determine a suitable word
$w_b$ for a given braid.  Again, we may allow $w=w_b x_1^k$ for any
$k$ (since we are not worried about extra twisting anymore).

After Abelianization, these relations become $x_1=(x_1')^{n_2}$, and
$t'=t^{n_1}w(x_i)$. At each level, we get a $\mathbb Z$ generated by
$x_1'$, and while $t'$ might not equal zero, it %is a word in things that can be made to
can be written as a word in $x_1'$ as the previous $t$ could be written as a word in $x_1$.  We note here that we can always take the first solid torus $T_0$ to be standardly embedded, so that the longitude $t_0=e$. This follows from a theorem of Alexander \cite{alexander}, which states that every knot (or link) can be represented as a closed braid.  Then the maps from $\mathbb Z \to \mathbb Z$ are again multiplication by $n_i$.  Thus the Abelianization of all these groups depends only on the solenoid, not the embedding.

The preceding fact is actually a simple consequence of Alexander duality:
\begin{thm}[Alexander Duality]\ulabel{alexander}{Theorem}
For a compact set $K\subset S^n$, $H_i(S^n-K)\cong \check{H}^{n-i-1}(K)$.
\end{thm}

In our setting, this tells us that the first homology, or the Abelianization of the fundamental group, of the complement of an embedded solenoid is equal to the first \v{C}ech cohomology of the solenoid, which is independent of the embedding: $(\pi_1)_{\text{Ab}}=H_1(S^3-\Sigma)=\check{H}^1(\Sigma)$.  Of course the \v{C}ech cohomology of $\Sigma(n_i)$ must then be the group $G(n_i)$ as discussed in the previous section, since in that case the fundamental group is the first homology group, being Abelian.  That this group is in fact the \v{C}ech cohomology of the solenoid is shown/discussed in \cite{mccord}.

\begin{ex}[Dyadic Solenoid]\ulabel{dyadicknotted}{Example}
Again, consider the dyadic solenoid with $n_i\equiv2$.  On each level we will use the braid ${\sigma_1}^3$, which gives the trefoil knot.  In this case the presentation for the fundamental group becomes:
$$\pi_1=\Big\langle s_i,z_i ~\Big|~ [s_i,z_i]=e,~ {s_i}^{-1}z_{i+1}s_i={z_{i}}^{-1}{z_{i+1}}^{-1}{z_i}^2,~ s_{i+1}={s_i}^2{z_i}^3{z_{i+1}}^{-6},~ s_0=e \Big\rangle
$$
%$$\pi_1=\Big\langle s_i,z_i ~\Big|~ [s_i,z_i]=e,~ s_i^{-1}z_{i+1} s_i = z_i^{-1}z_{i+1}^{-1}z_i^2,~
%s_{i+1}=s_i^2 z_i^3 z_{i+1}^{-6} ,~ s_0=e \Big\rangle
%$$

Note that if  we Abelianize, then $z_{i+1}^2=z_i$, and then $s_{i+1}=s_i^2=e$ as before.  This gives us that $H_1=(\pi_1)_{\text{Ab}}$ is the dyadic rationals.

However, this fundamental group is non-Abelian.  This follows from \ref{directlimit} and the fact that the trefoil group is non-Abelian.  This can also be seen directly from the presentation, as the fundamental group maps onto the infinite alternating group $A_{\infty}$.  To see this, map each generator $z_i$ to the 3-cycle $\big(i~(i+1)~(i+2)\big)$, and map each $s_i$ to the identity.  It is straightforward to check that the relations are satisfied in $A_\infty$; the only one of these that is not immediate follows since consecutive 3-cycles satisfy the relation  $z_{i+1}={z_{i}}^{-1}{z_{i+1}}^{-1}{z_i}^2$.
%As each $s_i$ is just a word in the $z_i$'s and the previous $s_j$'s, this defines the homomorphism.
\end{ex}

While the homology of a solenoid complement only depends on the
solenoid, the fundamental groups can be quite different.  However,
it is still difficult to tell them apart.  We have given a way to
present these groups, but our presentations are infinite, which
makes it difficult to determine when two groups are isomorphic; in fact it is even difficult to tell when two finite presentations give isomorphic groups.
For instance, if we take a dyadic solenoid with $n_i\equiv2$, at any level we may either use the unknotted embedding from \ref{dyadicunknot}, or the trefoil embedding from \ref{dyadicknotted}.  The presentation will look similar to those in the examples, using the relations from one or the other at different levels $i$ depending on which embedding was chosen.  While it seems very likely that these give different fundamental groups, it is hard to prove that for these given infinite presentations, especially as they have isomorphic Abelianizations (see \ref{alexander}).

However, despite these difficulties, we can tell some of these embeddings apart via the fundamental group.
\ref{directlimit} tells us that the fundamental groups of the various stages inject into the fundamental group of the entire complement. A standard result from knot theory states that the fundamental group of the complement of any knot other than the unknot is non-Abelian. Thus if there is any knotting in our embedding of the solenoid, $\pi_1(S^3-\Sigma)$ will be non-Abelian, in contrast to the unknotted embeddings which always have Abelian fundamental groups.

There are many knotted embeddings of any solenoid, which seemingly should all be different.  As fundamental groups determine knots (up to chirality), it seems that if there is any substantial difference in the knottings, the fundamental groups should differ.  Unfortunately, it is hard to show
this given our infinite presentations.

We summarize the results of this section in the following theorem and conjecture.

\begin{thm}\ulabel{knotted}{Theorem}
For every solenoid $\Sigma$, there are knotted embeddings
$\Sigma \subset S^3$, and such embeddings have $\pi_{1}(S^3 - \Sigma)$ non-Abelian.
These embeddings are inequivalent to unknotted embeddings, whose complements have Abelian fundamental groups.
\end{thm}

\begin{conj}\ulabel{knottedconj}{Conjecture}
If a solenoid is embedded in two `different' knotted ways, the
fundamental groups of the complements are different.
\end{conj}

\section{Distinguishing Non-Abelian Complements}\ulabel{sec-geometry}{section}
%...new stuff with Jessica...

As discussed in the previous sections, for any solenoid there is an embedding with a non-Abelian fundamental group, which is clearly not equivalent to the Abelian embeddings.  As knots are essentially determined by the fundamental group of their complements (up to an issue of chirality), it seems that unknotted embeddings of a solenoid that are knotted in different ways should give different fundamental groups.  Unfortunately, the result for knots does not easily carry over to solenoids, as the fundamental groups are now ascending unions of knot groups, and it is not clear whether two ascending unions could be equal in the end, yet differ at every finite stage.

In order to distinguish non-Abelian embeddings of a given solenoid, we consider the geometry of the complements.  A standard tool we will use is the JSJ-decomposition, cutting the manifold along incompressible tori.  As the JSJ-decomposition only applies to compact manifolds, we will generalize it to apply to a certain class of embeddings of solenoids.  The following statement is taken from Hatcher's notes on 3-manifolds \cite{hatcher_3mfld}, under the section on Torus Decomposition.

\begin{thm}[JSJ-Decomposition]\ulabel{JSJ-thm}{Theorem}
A compact irreducible orientable 3-manifold has a minimal collection of disjoint incompressible tori such that each component of the complement of the tori is either atoroidal or Seifert fibered.
This minimal collection is unique up to isotopy.
\end{thm}

To generalize this result for solenoid embeddings, we need to consider embeddings such that infinitely many of the `solid torus minus a braid' pieces are hyperbolic.  As long as the braid has at least 3 strands, this should generically be the case.  If there are only 2 strands, the piece will always be Seifert fibered.

\begin{prop}\ulabel{hyperbolicbraids}{Proposition}
Given $n\geq3$, there exist (at least) two $n$-braids $B(n,i)$ in a solid torus $T$ such that the complements $T-B(n,i)$ have distinct hyperbolic structures for $i=1,2.$
\end{prop}
\begin{proof}
An $n$-braid in a solid torus is the mapping torus of an $n$-punctured disk $B^2$.  Thurston \cite{thurston-bull, thurston-hypII} proves that such manifolds are hyperbolic precisely when the monodromy is pseudo-Anosov, and states that this is in fact the generic case (see Theorem 0.1 in \cite{thurston-hypII}).
\end{proof}

The proof above using Thurston's results only shows that
an $n$-braid in a solid torus will generically give a hyperbolic 3-manifold with 2 cusps, without constructing specific examples.  For a fixed choice of $n$, we can construct specific examples with different hyperbolic structures quite easily, and in \ref{table1} we present a few specific braids for $n=3,4,5$ in terms of the standard braid generators $\sigma_i$.  In general, it seems that the braids $\prod_{i=1}^{n-1} \sigma_i^{e_i}$, where $e_i=\pm1$, each give different volumes, unless there is either some obvious symmetry (i.e.\ $-++$ gives the same as $++-$, $+--$ and $--+$), or if it is Seifert fibered (i.e.\ $---$ or $+++$).  Of course, for $n=3$ we must add extra twisting, since there are only 2 generators $\sigma_i$, which only gives one hyperbolic 3-braid knot with two crossings, up to symmetry.
The hyperbolic volumes given in \ref{table1} were calculated using SnapPea \cite{weeks:snappea}.

\begin{table}
\begin{center}
\begin{tabular}{|c|c|c|}
\hline
$n$ & Braid & Hyperbolic Volume\\
\hline
3& $\sigma_1^{-1} \sigma_2$ & 4.05 \\
 & $\sigma_1^{-3} \sigma_2$ & 5.97 \\
\hline
4& $\sigma_1^{-1} \sigma_2 \sigma_3$ & 4.85 \\
 & $\sigma_1^{-1} \sigma_2 \sigma_3^{-1}$ & 7.51 \\
%\hline
%4& $\sigma_1^-3 \sigma_2 \sigma_3$ &  \\
%4& $\sigma_1^-3 \sigma_2 \sigma_3^-1$ &  \\
\hline
5& $\sigma_1^{-1} \sigma_2 \sigma_3 \sigma_4$ & 5.08 \\
 & $\sigma_1^{-1} \sigma_2^{-1} \sigma_3 \sigma_4$ & 5.90 \\
 & $\sigma_1^{-1} \sigma_2 \sigma_3^{-1} \sigma_4$ & 11.2 \\
\hline
\end{tabular}
\end{center}
\caption{Braids in a solid torus with distinct hyperbolic volumes.}\ulabel{table1}{Table}
\end{table}
%asdfasdf

Recall that hyperbolic structures on 3-manifolds are in fact topological invariants, as given by Mostow-Prasad rigidity \cite{mostow,prasad}:
\begin{thm}[Mostow-Prasad Rigidity]
If a 3-manifold admits a complete hyperbolic structure with finite volume, then that structure is unique up to isometry.
\end{thm}
Using Mostow-Prasad rigidity and \ref{hyperbolicbraids}, we are able to prove the existence of inequivalent non-Abelian embeddings for any given solenoid.

\begin{thm}\ulabel{thm-geometry}{Theorem}
For any solenoid, there exist uncountably many inequivalent non-Abelian embeddings, i.e.\ such that the complements are different manifolds.
\end{thm}

\begin{proof}
%We first assume that we do not have the dyadic solenoid.
Choose a defining sequence $n_i$ for the solenoid $\Sigma$, with the condition that $n_i\neq2$.  If necessary, we may take the product of consecutive terms $n_i$ to ensure that $n_i\neq2$.
%The only condition we require is that if $\Sigma$ is the dyadic solenoid then infinitely many of the $n_i\neq2$.
%% can we just do 4,4,4... for the dyadic?

We will construct different non-Abelian embeddings of $\Sigma$.  Let $T_0$ be a knotted solid torus with cross-sectional diameter 1 in $S^3$.  To the complement of $T_0$, glue in either $T-B(n_1,1)$ or $T-B(n_1,2)$, one of the hyperbolic manifolds from \ref{hyperbolicbraids}.  Continue attaching either $T-B(n_i,1)$ or $T-B(n_i,2)$. As we fill in the braids, make sure that the cross-sectional diameter of each braid is less than half the diameter of the previous level.  This will embed the solenoid $\Sigma(n_i)$.  As we have two choices at each stage, there are uncountably many ways of doing this.  It remains to show that these each give different complements.

We will use the JSJ-decomposition.
Take any incompressible torus $T^*$ in $S^3-\Sigma$.  This cuts $S^3$ into a compact piece and a noncompact piece, because $\Sigma$ is connected.  There is a small torus $T_k$ in our construction that lies inside the non-compact piece, as $T^*$ is bounded away from $\Sigma$, and we ensured that the tori $T_i$ had cross-sectional diameter less than $2^{-i}$.  This torus $T_k$ then cuts $S^3$ into two new pieces, again one compact and one not, with the originally chosen incompressible torus $T^*$ in the compact piece.  Now apply the JSJ-decomposition (\ref{JSJ-thm}) to the compact piece.  As the pieces $T^2-B(n,i)$ in our construction were chosen to be hyperbolic they are atoroidal, and thus the torus $T^*$ must be isotopic to one of our defining tori $T_i$.

Thus we get a canonical JSJ-decomposition of our solenoid complement, with every incompressible torus in the complement being isotopic to one of the defining tori.  In particular, the incompressible tori cut $S^3-\Sigma$ into pieces, one of which has one cusp (the original knot complement), and all the rest having 2 cusps.  These pieces may be ordered by taking the piece with one cusp as the first, and then considering which other pieces share a common boundary.  So we have a canonical way of cutting up the solenoid complement into these ordered pieces.  If any of the pieces are different at any spot in the sequence, the resulting manifolds are distinct, which proves the theorem.
\end{proof}

\begin{cor}
Let $\{n_i\}$ be any defining sequence of a solenoid, other than a sequence that is eventually 2 for the dyadic solenoid. Then there are uncountably many inequivalent embeddings of the solenoid using the sequence $n_i$.
\end{cor}
\begin{proof}
Proceed with the construction as in the proof of the theorem, except when $n_i=2$, fill in with any Seifert fibered 2-braid.  In fact, all we need is that infinitely many of the pieces are hyperbolic.  Then to get the generalized JSJ-decomposition, when given an incompressible torus $T^*$, choose the small torus $T_k$ such that $T_k$ represents the inner braid in one of the hyperbolic pieces.  Again we may apply the standard JSJ-decomposition to the compact complementary component of $T_k$. This gives us that $T^*$ is either one of our defining tori $T_i$, or that $T^*$ lies in one of the Seifert fibered pieces.

Again, we get a canonical JSJ-decomposition, where on each compact piece we take the minimal collection of tori guaranteed by the standard JSJ-decomposition.  As we have infinitely many hyperbolic pieces, as we can choose to fill in with non-isometric pieces, we get uncountably many distinct complements.
\end{proof}

Note that this proof cannot be extended to the defining sequence $n_i\equiv2$, as the homeomorphism type of a solid torus minus any 2-braid is only dependent on the number of components.  As we have only been considering knots, we will always have 1 component, thus one homeomorphism type of a solid torus minus a 2-braid.

\appendix

\section{Subgroups of $\Q$}%{Solenoid Stuff Appendix}
%\pagebreak
This lemma characterizes the additive subgroups of the rational numbers.  We note that these subgroups were previously discussed and characterized in \cite{beaumont,baer}, but we give our own proof here.
%%, and was proven by \cite{someone}
Note that for additive subgroups of $\Q$, multiplication by a constant is an isomorphism, so that we may assume that the subgroup contains 1.  In the lemma, the numbers $k_i$ represent the number of times (plus 1) that the prime $p_i$ is allowed to appear in the denominators of the subgroup elements.

\begin{lem}\ulabel{subgroups of Q}{Lemma}
Let $\{k_i\}$ be a sequence in $\mathbb N\cup\infty$.  Define
$$Q\big(\{k_i\}\big)=\left\{\frac pq\in\mathbb Q ~\middle|~ q =%\text{ divides }
\prod_{i=1}^m
p_i^{n_i} \quad\text{ for some } n_i< k_i \text{ and some } m %\right{\phantom{)}}   %
\right\}
$$
where $p_i$ denotes the $i^{th}$ prime number.

Then $Q\big(\{k_i\}\big)$ is a subgroup of $\mathbb Q$ containing 1. Furthermore,
every subgroup $G\leq\mathbb Q$ containing 1 is equal to $Q\big(\{k_i\}\big)$
for some sequence $\{k_i\}$.
\end{lem}

\begin{proof}
Since the definition does not require the fraction $p/q$ to be in lowest terms, $Q\big(\{k_i\}\big)$ is clearly closed under addition and inverses, and is thus a subgroup containing 1.

Let $Q$ be any subgroup of $\Q$ containing 1.  Let $D$ be the set of denominators of elements of $Q$ when written in lowest terms, i.e.\ $D=\{q~|~p/q\in Q \text{ in lowest terms}\}$.  Note that for every $q\in D$, we must have $1/q\in Q$, since $p/q\in Q$ with $(p,q)=1$, so that if we multiple $p/q$ by the multiplicative inverse of $m\mod q$ we get $mp/q=M+1/q$.  Since $1\in Q$, then $1/q\in Q$.  Then also $a/q\in Q$ for every $a\in\Z$ and $q\in D$, and in fact $Q$ is the set of all such numbers $\{a/q\}$, as every element of $Q$ is equal to a reduced fraction with denominator $q\in D$.

Define the number $k_i\in\N\cup\infty$ to be one more than the maximum number of times the prime $p_i$ appears in an element of $D$; $k_i=\sup\big\{ 1+k~\big|~ {p_i}^k \text{ divides } q \text{ for some } q\in D \big\}$. We first show that $Q\subset Q\big(\{ k_i \}\big)$.  Let $a/q\in Q$, where $q\in D$.  Consider the prime factorization $q =\prod_{i=1}^m p_i^{n_i}$, where $n_i<k_i$ by the definition of $k_i$.  Thus $a/q\in Q\big(\{ k_i \}\big)$ for every $a/q\in Q$.

It remains to show $Q\big(\{ k_i \}\big)\subset Q$.  Note that $Q\big(\{ k_i \}\big)$ is generated by elements of the form $ 1\big/\prod^m p_i^{n_i}$.  In fact, we can take elements of the form $1/p_i^{n_i}$ as our generating set: since the $p_i^{n_i}$ are relatively prime, we may choose $a_i$ so that $\sum (a_i/p_i^{n_i})=1\big/\prod^m p_i^{n_i}$.
Thus it suffices to show that $1/p_i^{n_i}\in Q$ if $n_i<k_i$.  By the definition of $k_i$, we know that there is an element $a/(bp_i^{n_i})\in Q$ in reduced form.  As before, since $a$ is relatively prime to the denominator $q$, we may multiply by the inverse of $a \mod q$ and thus assume that $a=1$.  Then multiplying by $b$ gives $1/p_i^{n_i}\in Q$.

Therefore every subgroup of $\Q$ is of the form $Q\big(\{k_i\}\big)$ for some sequence $\{k_i\}$.
\end{proof}

We note that while different sequences $\{k_i\}$ give distinct subsets of
$\Q$, they do not always give non-isomorphic subgroups.  This is due to the
fact that multiplication gives isomorphisms of subgroups of $\Q$.  Thus if
two sequences $\{k_i\},\{k_i'\}$ differ in only finitely many spots by a
finite amount (i.e.\ if whenever $k_i\neq k_i'$ then both are finite), then
the subgroups are isomorphic by multiplication/division by $\prod
p_i^{k_i-k_i'}$.  This is in fact the only way differing sequences can give
isomorphic groups.
\\

%\begin{acknowledgements}
\begin{center}
\textsc{Acknowledgements}\\[10pt]
\end{center}

We would like to thank Jessica Purcell for helpful ideas and discussions. % on this paper.
%\end{acknowledgements}
\\

\bibliographystyle{plain}
%% the name of your bib file without the .bib extension. For example, if your file is thesis.bib,
%% you would put \bibliography{thesis}
\bibliography{refs}

\end{document}